\documentclass[12pt]{amsart}

\usepackage[cp1251]{inputenc}
\usepackage{amsthm,amssymb,amsmath,amsfonts,wasysym}
\usepackage{textcomp} 
\usepackage{etoolbox}
\usepackage{comment}
\usepackage{graphicx}
\usepackage[matrix,arrow,curve]{xy}
\usepackage{longtable}
\newcommand{\LL}{\ensuremath{\mathbb{L}}}

\newcommand{\F}{\ensuremath{\mathbb{F}}}

\newcommand{\ka}{\ensuremath{\Bbbk}}

\newcommand{\kka}{\ensuremath{\overline{\Bbbk}}}

\newcommand{\XX}{\ensuremath{\overline{X}}}

\newcommand{\Pro}{\ensuremath{\mathbb{P}}}

\newcommand{\Am}{\ensuremath{\operatorname{Am}}}

\newcommand{\Aut}{\ensuremath{\operatorname{Aut}}}

\newcommand{\Gal}{\ensuremath{\operatorname{Gal}}}

\newcommand{\Pic}{\ensuremath{\operatorname{Pic}}}

\makeatletter
\@addtoreset{equation}{section}
\makeatother

\newtheorem{theorem}[equation]{Theorem}
\newtheorem{proposition}[equation]{Proposition}
\newtheorem{lemma}[equation]{Lemma}
\newtheorem{corollary}[equation]{Corollary}

\theoremstyle{definition}
\newtheorem{example}[equation]{Example}
\newtheorem{definition}[equation]{Definition}

\theoremstyle{remark}
\newtheorem{remark}[equation]{Remark}
\newtheorem{notation}[equation]{Notation}
\newtheorem{question}[equation]{Question}

\title{Birational classification of pointless del Pezzo surfaces of degree $8$}

\address{Steklov Mathematical Institute of Russian Academy of Sciences, 8 Gubkina st., Moscow, 119991, Russia}
\address{Laboratory of Algebraic Geometry, National Research University Higher School of Economics, 6 Usacheva str., Moscow 119048, Russia}
\email{trepalin@mccme.ru}
\thanks{This work is supported by the Russian Science Foundation under grant \textnumero 18-11-00121.}
\author{Andrey Trepalin}

\begin{document}

\begin{abstract}

Let $\ka$ be a perfect field. Recently \mbox{J.-L.\,Colliot-Th\'el\`ene} showed that two pointless quadric surfaces over $\ka$ are birationally equivalent if and only if they are isomorphic. We show that this result holds for arbitrary del Pezzo surfaces of degree $8$ with the Picard number $1$, and describe minimal surfaces birationally equivalent to a given pointless del Pezzo surface of degree $8$.

\end{abstract}



\maketitle

\section{Introduction}

Let $\ka$ be a perfect field, and $Q$ be a smooth quadric surface in $\Pro^3_{\ka}$. If $\ka$ is~algebraically closed then all smooth quadrics are isomorphic, but if $\ka$ is not algebraically closed then two quadrics can be non-isomorphic. The biregular classification of quadric surfaces over nonclosed fields is well-known (see Section 2 for the details). Moreover, recently in~\cite{CT21} J.-L.\,Colliot-Th\'el\`ene showed that the birational classification is almost the~same.

Note that if there is a $\ka$-point on a quadric surface $Q$ in $\Pro^3_{\ka}$, then one can consider a~projection from this point to $\Pro^2_{\ka}$. This projection is a birational map $Q \dashrightarrow \Pro^2_{\ka}$. Therefore any quadric with a $\ka$-point is $\ka$-rational. In particular, any two such quadrics are birationally equivalent over $\ka$. But if there are no $\ka$-points on a quadric then it is obviously not $\ka$-rational. The birational classification of pointless quadric surfaces is given in~the~following theorem.

\begin{theorem}[{\cite{CT21}}]
\label{main1}
Let $\ka$ be an algebraically nonclosed field such that $\operatorname{char} \ka \neq 2$, and~let~$Q_1$ and~$Q_2$ be two pointless quadrics in $\Pro^3_{\ka}$. Then $Q_1$ is birationally equivalent to~$Q_2$ if~and~only if $Q_1 \cong Q_2$.

\end{theorem}

The aim of this paper is to obtain an analogue of Theorem \ref{main1} for the case of~arbitrary pointless del Pezzo surfaces of degree $8$. Let us recall the definition of a del Pezzo surface.

\begin{definition}
\label{DPdef}
A {\it del Pezzo surface} is a smooth projective surface~$X$ such that the~anticanonical class $-K_X$ is ample. The number $d = K_X^2$ is called the {\it degree} of a del Pezzo surface $X$.
\end{definition}

Any quadric in $\Pro^3_{\ka}$ is a del Pezzo surface of degree $8$. Moreover, over algebraically closed field any del Pezzo surface of degree $8$ is either isomorphic to $\Pro^1_{\ka} \times \Pro^1_{\ka}$, or isomorphic to~the~blowup of $\Pro^2_{\ka}$ at a point. Therefore for a perfect field $\ka$ and a del Pezzo surface $X$ of degree $8$ there are two possibilities: either $\XX = X \otimes_{\ka} \kka$ is~isomorphic to $\Pro^1_{\kka} \times \Pro^1_{\kka}$, or~$\XX$ is isomorphic to the blowup of $\Pro^2_{\kka}$ at a point. In the former case if $X(\ka) \neq \varnothing$ then~$X$ is~isomorphic to a quadric in $\Pro^3_{\ka}$. In the latter case there is a unique $(-1)$-curve on $\XX$. This curve is defined over $\ka$, therefore we can contract this curve and get a~morphism~${X \rightarrow \Pro^2_{\ka}}$. In particular, one has $X(\ka) \neq \varnothing$. Thus for any pointless del Pezzo surface $X$ of degree $8$ the surface $\XX$ is isomorphic to~$\Pro^1_{\kka} \times \Pro^1_{\kka}$. 

If $X$ is a pointless del Pezzo surface of degree $8$ then $\Pic\left(\XX\right) = \Pic\left(\Pro^1_{\kka} \times \Pro^1_{\kka}\right) \cong \mathbb{Z}^2$. Therefore the Picard number $\rho(X) = \operatorname{rk} \Pic(X)$ is equal to $1$ or $2$, since $\Pic\left(X\right) \subset \Pic\left(\XX\right)$.

To establish more results about birational classification of pointless del Pezzo surfaces of degree $8$ we need the following definition.

\begin{definition}
\label{minimality}
A smooth projective surface $S$ is called {\it minimal} if any birational morphism of~smooth surfaces $S \rightarrow S'$ is an isomorphism.
\end{definition}

The classification of minimal geometrically rational surfaces is well-known.

\begin{theorem}[{cf. \cite[Theorems 1, 4, 5]{Isk79}}]
\label{Minclass}
Let $S$ be a minimal geometrically rational surface. Then either $S$ admits a conic bundle structure over a~conic with~$\rho(S) = 2$, or $S$ is a del Pezzo surface with $\rho(S) =1$. In the former \mbox{case $K_S^2 \notin \{3, 5, 6, 7\}$,} and for $K_S^2 = 8$ the surface $S$ is not isomorphic to the blowup of $\Pro^2_{\ka}$ at a point.
\end{theorem}

Note that any pointless del Pezzo surface $X$ of degree $8$ is minimal. If $\rho(X) = 1$ then it is obvious, and if $\rho(X) = 2$ then $X$ is isomorphic to $C_1 \times C_2$, where $C_1$ and $C_2$ are smooth conics, and the projections $X \rightarrow C_1$ and $X \rightarrow C_2$ define two structures of conic bundles on $X$.

We want to study which surfaces are birationally equivalent to pointless del Pezzo surfaces of degree $8$. Obviously, for any variety one can equivariantly blow up a collection of points defined over $\ka$ and get a birationally equivalent variety. Therefore for a given pointless del Pezzo surface $X$ of degree $8$ the natural problem is to find minimal surfaces birationally equivalent to $X$.

For the case $\rho(X) = 2$ the answer seems to be well-known to experts. For example, many results for this case can be found in \cite{Kol05}. For completeness we prove the following theorem.

\begin{theorem}
\label{mainrho2}
Let $\ka$ be a perfect field, and let $X$ be~a~pointless del Pezzo surface of~degree~$8$ such that $\rho(X) = 2$. Then $X$ is a product of~two~smooth conics and there are two following possibilities.

\begin{enumerate}

\item The surface $X$ is isomorphic to $C_1 \times C_2$, where $C_1$ and $C_2$ are two non-trivial smooth conics not isomorphic to each other. If the Brauer product $C_3 = C_1 * C_2$ (for the definition see Definition \ref{BPdef} or \cite[7]{Kol05}) is defined, then any minimal surface birationally equivalent to $X$ is isomorphic to $C_1 \times C_2$, $C_1 \times C_3$ or $C_2 \times C_3$. Otherwise, any minimal surface birationally equivalent to $X$ is isomorphic to $X$.

\item The surface $X$ is isomorphic to $C \times C$ or $C \times \Pro^1_{\ka}$, where $C$ is a non-trivial smooth conic. Then any minimal surface birationally equivalent to $X$ is isomorphic \mbox{to $C \times C$,} $C \times \Pro^1_{\ka}$ or to a $\ka$-form of Hirzebruch surface $\mathbb{F}_{2k}$ admitting a conic bundle structure over $C$. Such $\ka$-form is unique up to isomorphism for a given $k$.

\end{enumerate}

\end{theorem}

For the case $\rho(X) = 1$ we prove the following theorem.

\begin{theorem}
\label{mainrho1}
Let $\ka$ be a perfect field, and let $X$ be~a~pointless del Pezzo surface of~degree~$8$ such that $\rho(X) = 1$. Then any minimal surface birationally equivalent to~$X$ is~isomorphic to $X$.

\end{theorem}

A geometrically rational surface $X$ admitting a structure of a conic bundle is called a~\textit{relatively minimal conic bundle} if the Picard number of this surface is $2$.

Let $X$ be a minimal del Pezzo surface with $\rho(X) = 1$ or a relatively minimal conic bundle. The surface $X$ is called \textit{birationally rigid} if for any birational \mbox{map $X \dashrightarrow X'$,} where $X'$ is a minimal del Pezzo surface $X'$ with $\rho(X') = 1$ or a relatively minimal conic bundle, one has $X' \cong X$ (for the precise definition of~birational rigidity in general case and some properties see e.g. \cite{Ch05}).

In Section $4$ we show that in general case a pointless del Pezzo surface $X$ of degree~$8$ with $\rho(X) = 1$ is birationally equivalent to a (non-minimal) del Pezzo surface of degree~$6$ admitting a structure of relatively minimal conic bundle. For a given surface $X$ these surfaces are parametrised by quadratic extensions $F/\ka$ such that $X_F(F) \neq \varnothing$. Therefore~$X$ is not birationally rigid.

Let us recall that the \textit{index} $I(V)$ of a variety $V$ is the greatest common divisor of~the~degrees of closed points on $V$. We give a description of birationally rigid del Pezzo surfaces of~degree~$8$ in the following theorem.

\begin{theorem}
\label{rigidity}
Let $\ka$ be a perfect field, and let $X$ be a del Pezzo surface of degree $8$. Then the following assertions are equivalent.

\begin{itemize}
\item[(a)] One has $I(X) = 4$.

\item[(b)] There are no points of degree $2$ on $X$.

\item[(c)] The Amitsur subgroup $\Am(X) \subset \operatorname{Br}(\ka)$ (see Definition \ref{Amitsur} or \cite[Definition~2.8]{Lied17}) contains an element that does not correspond to a conic.

\item[(d)] The surface $X$ is birationally rigid.

\end{itemize}

\end{theorem}

Actually for any pointless del Pezzo surface $X$ of degree $8$ we describe all possible minimal del Pezzo surfaces and relatively minimal conic bundles birationally equivalent to $X$. Note that if $X$ is a minimal del Pezzo surface $X$ with $\rho(X) = 1$ of degree $1$, $2$ or~$3$, or a pointless del Pezzo surface of degree $4$, then $X$ is birationally rigid (see \cite[Theorems 4.4 and 4.5]{Isk96}). On the other hand by \cite[Chapter 4]{Isk96} if $X$ is a del Pezzo surface of degree at least $5$ with $\rho(X) = 1$ and $X(\ka) \neq \varnothing$  or $X \rightarrow \Pro^1_{\ka}$ is a conic bundle with~$\rho(X) = 2$, $X(\ka) \neq \varnothing$ and $K_X^2 \geqslant 5$ then $X$ is $\ka$-rational. In particular, all these surfaces are birationally equivalent to each other.

A non-trivial Severi--Brauer surface $X$ (i.e. a pointless del Pezzo surface of degree~$9$) is not birationally equivalent to any conic bundle, and if a minimal del Pezzo surface~$X'$ is birationally equivalent to $X$ then either $X' \cong X$, or $X' \cong X^{\mathrm{op}}$, where $X^{\mathrm{op}}$ \mbox{is the Severi--Brauer} surface such that the central simple algebras corresponding to $X$ and $X^{\mathrm{op}}$ are opposite (see \cite[Corollary 2.4 and Theorem 2.10]{Sh20} and \cite{Wei22}).

Note that a geometrically rational surface $X$ with $K_X^2 = 7$ is never minimal by \cite[Theorem 4]{Isk79}, and a del Pezzo surface of degree $5$ always has a $\ka$-point (see \cite{SD72}).

An interesting problem is to describe all minimal surfaces birationally equivalent to~a~given del Pezzo surface $X$ with $\rho(X) = 1$, or a conic bundle $X \rightarrow C$ with $\rho(X) = 2$. The cases of pointless del Pezzo surfaces of degree $8$ and pointless conic bundles with~$0$ or~$2$ degenerate fibres are considered in this paper. Therefore the remaining cases are conic bundles with at least $4$ degenerate fibres, pointless del Pezzo surfaces of degree $6$ and del Pezzo surfaces of degree $4$ with $\ka$-points. One of the particular questions is whether there exist two birationally equivalent del Pezzo surfaces $X$ and $X'$ of degree $4$, such that $X$ and $X'$ are not isomorphic. Note that some useful results about birational classification of minimal conic bundles and minimal del Pezzo surfaces of degree $4$ with $\ka$-points are obtained in \cite[Theorem 2.5 and Corollary 3.3]{Sk86}.

The plan of this paper is as follows. In Section $2$ we recall some notions and properties of pointless del Pezzo surfaces of degree $8$ and give a definition of Amitsur \mbox{subgroup $\Am(X) \subset \operatorname{Br}\left(\ka\right)$} (see Definition \ref{Amitsur}), that is a birational invariant. Also we give a description of~all~geometrically rational surfaces with non-trivial $\Am(X)$, and show how one can restore a del Pezzo surface $X$ of~degree $8$ with~$\rho(X) = 1$ by~$\Am(X_{\LL})$, where~$\LL$ is a~\textit{splitting field} of $X$: unique quadratic extension of~$\ka$ such that $\rho\left(X_{\LL}\right) = 2$.

In Section $3$ we consider pointless del Pezzo surfaces of degree $8$ with the Picard number~$2$. Any such surface $X$ is isomorphic to a product of two conics. We consider Sarkisov links for these surfaces, describe possibilities for minimal surfaces birationally equivalent to $X$, and show that any other surface is not birationally equivalent to $X$. As a result of~this section we prove Theorem \ref{mainrho2}.

In Section $4$ we consider pointless del Pezzo surfaces of degree $8$ with the Picard number~$1$. We show that for any such surface $X$ any Sarkisov link or sequence of Sarkisov links leads to an isomorphic surface or a certain non-minimal del Pezzo surface of degree~$6$ obtaining a relatively minimal conic bundle structure. Finally, we prove Theorems \ref{mainrho1} and \ref{rigidity}, and give an alternative proof of Theorem \ref{main1} for the case of a perfect field.

\smallskip
\textbf{Acknowledgements.}
The author is grateful to Sergey Gorchinskiy for his advice to~use the~group $\Am(X)$ for studying the birational classification of del Pezzo surfaces of~degree $8$. The author is grateful to Costya Shramov for numerous discussions during the~preparation and writing of this paper. Also the author thanks \mbox{Jean-Louis\,Colliot-Th\'el\`ene} for many useful discussions, that were helpful to improve the~results of this paper. Also the author would like to thank the reviewer of this paper for useful comments.
\begin{notation}

Throughout this paper $\ka$ is a perfect field, $\kka$ is its algebraic closure, and~the~Galois group $\Gal\left( \kka / \ka\right)$ is denoted by $G_{\ka}$. For a surface $X$ we denote~$X \otimes \kka$ by $\XX$. For a surface $X$ we denote the Picard group  by $\Pic(X)$. \mbox{The number $\rho(X) = \operatorname{rk} \Pic(X)$} is the Picard number of $X$. If two surfaces $X$ and $Y$ are birationally equivalent then we write~$X \approx Y$. If two divisors $A$ and~$B$ are linearly equivalent then we write~$A \sim B$. The rational ruled (Hirzebruch) surface $\Pro_{\Pro^1_{\kka}}\left( \mathcal{O} \oplus \mathcal{O}(n) \right)$ is denoted by $\F_n$. For a given quadratic extension $\LL/\ka$ and a conic $C$ over $\LL$ we denote by $R_{\LL/\ka} C$ its Weil restriction of~scalars.

\end{notation}

\section{Preliminaries}

In this section we review some results about biregular classification of del Pezzo surfaces of degree $8$, and give an alternative description of this classification in terms of the Brauer group.

Throughout this section we assume that $X$ is a del Pezzo surface of degree $8$ over a~perfect field $\ka$, such that $\XX \cong \Pro^1_{\kka} \times \Pro^1_{\kka}$. The following lemma gives biregular classification of such surfaces.

\begin{lemma}[cf. {\cite[Lemma 3.4(i and ii)]{ShV18}}]
\label{Class1}
Let $\ka$ be a perfect field, and~let~$X$ be~a~del Pezzo surface of degree $8$, such that $\XX \cong \Pro^1_{\kka} \times \Pro^1_{\kka}$. The following assertion hold.

\begin{enumerate}

\item[(i)] The surface $X$ is isomorphic to the product $C_1 \times C_2$ of two conics over $\ka$, or~to~$R_{\LL/\ka} C$, where $\LL$ is a quadratic extension of $\ka$ and $C$ is a conic over $\LL$. In the former case one has $\rho(X) = 2$, while in the latter $\rho(X) = 1$. Furthermore, in the former case the (non-ordered) pair of conics $\{C_1, C_2\}$ is uniquely determined by $X$; in the latter case the extension $\LL / \ka$ and the conic $C$ are uniquely determined by $X$ up to conjugation by the Galois group $\Gal\left(\LL / \ka\right)$.

\item[(ii)] The surface $X$ is isomorphic to a quadric in $\Pro^3_{\ka}$ if and only if $X \cong C \times C$ \mbox{or $X \cong R_{\LL/\ka} N_{\LL}$,} for some conic $N$ over $\ka$.

\end{enumerate}

\end{lemma}

\begin{corollary}
\label{Class2}
Let $\ka$ be a perfect field, and let $X$ be a del Pezzo surface of~degree $8$ with~${\rho(X) = 1}$. In this case $X \cong R_{\LL/\ka} C$. Let $C'$ be a $\Gal\left(\LL / \ka\right)$-conjugate conic of $C$. Then $X_{\LL} \cong C \times C'$. In particular, $X$ is pointless if and only if $X_{\LL}$ is pointless.
\end{corollary}

For the both cases $\rho(X) = 2$ and $\rho(X) = 1$ we want to give an interpretation of~Lemma~\ref{Class1} in terms of the Brauer group~$\operatorname{Br}\left(\ka\right)$ (for the definition of the Brauer group and properties see, for example \cite[Chapter 3]{GS18}). Let us recall some facts about this group.

\begin{proposition}[{\cite[Proposition 5.1]{CTKM08}}]
\label{Brauer}
Let $V$ be a smooth projective geometrically irreducible variety over $\ka$. Then there exists an exact sequence
$$
0 \rightarrow \Pic(V) \rightarrow \Pic(\overline{V})^{G_{\ka}} \rightarrow \operatorname{Br}\left(\ka\right) \rightarrow \operatorname{Br}\left(\ka(V)\right).
$$   
\end{proposition}

\begin{definition}[see {\cite[Definition 2.8]{Lied17}}]
\label{Amitsur}
Let $V$ be a smooth projective geometrically irreducible variety over $\ka$. The group 
$$
\Am(V) = \Pic(\overline{V})^{G_{\ka}} / \Pic(V) \subset \operatorname{Br}\left(\ka\right)
$$
is called \textit{Amitsur subgroup} of $V$ in $\operatorname{Br}\left(\ka\right)$.
\end{definition}

Note that~$\Am(V)$ is~a~birational invariant since it is the kernel \mbox{of the map $\operatorname{Br}\left(\ka\right) \rightarrow \operatorname{Br}\left(\ka(V)\right)$} (see \cite[Proposition 2.10]{Lied17}).

We want to describe minimal geometrically rational surfaces $S$ with non-trivial $\Am(S)$. These surfaces are described in \cite[Proposition 5.3]{CTKM08}, but we obtain a more detailed description. By~Theorem~\ref{Minclass} the surface $S$ is either a del Pezzo surface with $\rho(S) = 1$, or~admits a~conic bundle structure over a smooth conic with $\rho(S) = 2$.

Let us start from the case of conic bundle. We say that a conic is \textit{non-trivial} if it is not isomorphic to $\Pro^1_{\ka}$. Note that for each conic $C$ we can define a \textit{class} $b(C)$, that is an element in $2$-torsion subgroup of the Brauer group $\operatorname{Br}\left(\ka\right)$ (see, for example, \cite[Section 3.3]{GS18}). If two conics are not isomorphic then they have different classes, and $b(C)$ is trivial if and only if $C$ is trivial. Moreover, for a conic $C$ the class $b(C)$ is~a~generator of $\Am(C)$.

If for a smooth projective geometrically irreducible variety $V$ there exists \mbox{a map $V \dashrightarrow C$,} where $C$ is a smooth conic, then this map induces embeddings \mbox{$\Pic(C) \hookrightarrow \Pic(V)$} and $\Pic(\overline{C})^{G_{\ka}} \hookrightarrow \Pic(\overline{V})^{G_{\ka}}$, since we can consider the preimage of~a~general geometric point on $C$. Therefore $\Am(C) \subset \Am(V)$. In particular, if $C$ is non-trivial then $\Am(V)$ is non-trivial, since it contains $b(C)$.

Now we can describe conic bundles $S \rightarrow C$ with non-trivial $\Am(S)$.

\begin{proposition}
\label{ClassCB}

Let $S$ be a surface admitting a conic bundle structure $S \rightarrow C$ over a~smooth conic $C$ over a perfect field $\ka$, such that $\rho(S) = 2$ and $\Am(S)$ is~non-trivial. Then there are five possibilities.

\begin{enumerate}

\item The surface $S$ is isomorphic to $C_1 \times C_2$, where $C_1$ and $C_2$ are two non-trivial smooth conics not isomorphic to each other. The group $\Am(S) \cong \left(\mathbb{Z} / 2\mathbb{Z}\right)^2$ is~generated by~$b(C_1)$ and $b(C_2)$.

\item The surface $S$ is isomorphic to $C \times C$, where $C$ is a non-trivial smooth conic. The~group $\Am(S) \cong \mathbb{Z} / 2\mathbb{Z}$ is generated by $b(C)$.

\item The surface $S$ is isomorphic to $C \times \Pro^1_{\ka}$, where $C$ is a non-trivial smooth conic. The group $\Am(S) \cong \mathbb{Z} / 2\mathbb{Z}$ is generated by $b(C)$.

\item The surface $S$ is a $\ka$-form of Hirzebruch surface $\mathbb{F}_{2k}$ admitting a conic bundle structure over a smooth non-trivial conic $C$, where $k$ is a positive integer. \mbox{The group $\Am(S) \cong \mathbb{Z} / 2\mathbb{Z}$} is generated by $b(C)$.

\item The surface $S$ admits a conic bundle structure over a smooth non-trivial conic $C$ with $2k$ degenerate geometric fibres, where $k$ is a positive integer. \mbox{The group $\Am(S) \cong \mathbb{Z} / 2\mathbb{Z}$} is generated by $b(C)$.

\end{enumerate}

\end{proposition}

\begin{proof}

Note that $\rho\left(\overline{S}\right)^{G_{\ka}} = \rho\left(S\right) = 2$. Therefore $\Pic\left(S\right)$ is a sublattice of $\Pic\left(\overline{S}\right)^{G_{\ka}}$ of~finite index.

We start from the case, when the conic bundle $S \rightarrow C$ has a degenerate geometric fibre. Let us show that $\Pic\left(\overline{S}\right)^{G_{\ka}}$ is generated by $-K_{\overline{S}}$ and the class $F$ of a geometric fibre. Otherwise there exists a class $D$ in $\Pic\left(\overline{S}\right)^{G_{\ka}}$, such that $mD \sim -K_{\overline{S}} + nF$. Let~$E$ be~an~irreducible component of a geometric degenerate fibre. Then $E$ is a $(-1)$-curve, and one has $-K_{\overline{S}} \cdot E = 1$ and $F \cdot E = 0$. Therefore
$$
m(D \cdot E) = mD \cdot E = \left( -K_{\overline{S}} + nF \right) \cdot E = 1.
$$
It is possible if and only if $m = \pm 1$, but in this case $D$ lies in the lattice generated by~$-K_{\overline{S}}$ and $F$.

Note that $\Pic\left(S\right)$ contains $-K_S$ and $2F$, since there is a point of degree $2$ on $C$, and~a~fibre over this point has class $2F$. Therefore $\Am(S)$ is non-trivial if and only if~$\Pic\left(S\right)$ does not contain $F$, and for this case $\Am(S) \cong \mathbb{Z} / 2\mathbb{Z}$ is generated by $b(C)$. This is possible if and only if $C$ is a non-trivial conic. Also note that if $S \rightarrow C$ has degenerate fibres over odd number of geometric points, then $C \cong \Pro^1_{\ka}$, since there is a point of odd degree on $C$.

Now assume that the conic bundle $S \rightarrow C$ does not have degenerate geometric fibres. Then either $\overline{S} \cong \Pro^1_{\kka} \times \Pro^1_{\kka}$, or $S$ is a $\ka$-form of a Hirzebruch surface $\mathbb{F}_m$, \mbox{where $m \geqslant 1$.} In~the~latter case there is a unique section $H$ of $S \rightarrow C$ such that $H^2 = -m$. \mbox{The group $\Pic\left(\overline{S}\right)^{G_{\ka}} = \Pic\left(\overline{S}\right)$} is generated by the class of $H$ and a class $F$ of a geometric fibre. The group $\Pic\left(S\right)$ contains $2F$ and the class of $H$, since $H$ is unique and~therefore defined over $\ka$. Therefore $\Am(S)$ is non-trivial if and only if $\Pic\left(S\right)$ does not contain $F$, and for this case one has $\Am(S) \cong \mathbb{Z} / 2\mathbb{Z}$. This is possible if and only if $C$ is a non-trivial conic. Also note that $-K_S \sim 2H + (m+2)F$, therefore $mF \in \Pic\left(S\right)$, and if $\Am(S)$ is~non-trivial then $m$ is even.  

Now assume that $S \cong C_1 \times C_2$ and $\overline{S} \cong \Pro^1_{\kka} \times \Pro^1_{\kka}$. Note that $\Pic\left( \Pro^1_{\kka} \times \Pro^1_{\kka} \right)$ is generated by the classes $A$ and $B$ of fibres of the projections on the first and the second factors of~$\Pro^1_{\kka} \times \Pro^1_{\kka}$. If $\rho(S) = 2$ then $\Pic(S)$ contains $2A$ and $2B$. Therefore if $\Am(S)$ is non-trivial then $\Pic(S)$ is $\langle 2A, 2B \rangle$, $\langle 2A, A+B \rangle$, $\langle 2A, B \rangle$ or $\langle A, 2B \rangle$.

If $\Pic(S)$ contains $A$ or $B$ then $C_1$ or $C_2$ respectively is isomorphic to $\Pro^1_{\ka}$, and $S \cong C \times \Pro^1_{\ka}$, where $C$ is a non-trivial smooth conic. For this case $\Am(S) \cong \mathbb{Z} / 2\mathbb{Z}$ is generated by $b(C)$.

Let us show that $\Pic(S)$ contains $A + B$ if and only if $S$ is isomorphic to a quadric in~$\Pro^3_{\ka}$. If $A + B$ lies in $\Pic(S)$, then the linear system $|A+B|$ defines \mbox{an embedding $S \hookrightarrow \Pro^3_{\ka}$} and~the~image is a quadric surface. Conversely, if $S$ is isomorphic to a quadric in $\Pro^3_{\ka}$ then any hyperplane section of this quadric has class $A + B$ in $\Pic(S)$. Therefore for this case we can apply Lemma \ref{Class1} and get that $S$ is isomorphic to $C \times C$, where $C$ is a smooth conic over $\ka$. If $\Am(S)$ is non-trivial then $C$ is non-trivial, and for this case $\Am(S) \cong \mathbb{Z} / 2\mathbb{Z}$ is~generated by $b(C)$.

If $\Pic(S)$ does not contain $A$, $B$ and $A + B$, then $C_1$ and $C_2$ are non-trivial and~not~isomorphic to each other. For this case $\Am(S) \cong \left(\mathbb{Z} / 2\mathbb{Z}\right)^2$ is generated by~$b(C_1)$ and~$b(C_2)$.
\end{proof}

The surfaces considered in cases $(1)$--$(4)$ of Proposition \ref{ClassCB} widely appear in this paper. In the following example we show how to construct a surface considered in case $(5)$. For~this aim we slightly modify the constuction of an exceptional conic bundle (see \cite[Subsection 5.2]{DI09}).

\begin{example}
Let $C$ be a non-trivial conic, $p_1$, \ldots $p_{2k}$ be a collection of $2k$ geometric points defined over $\ka$, and $H \rightarrow C$ be a double cover of $C$ branched into this collection of points. Assume that $\mathbb{F}$ is a minimal field, such that all points $p_i$ are defined over~$\mathbb{F}$. Consider two conjugate geometric points $q_1$ and $q_2$ on $\Pro^1_{\ka}$ defined over a quadratic extension~$\LL / \ka$ and not defined over $\mathbb{F}$.

Let $\iota_1$ be the involution of the double cover $H \rightarrow C$, and $\iota_2$ be an involution acting on $\Pro^1_{\ka}$ such that $\iota_2 \left( q_i \right) = q_i$. Consider an involution $\iota$ on $H \times \Pro^1_{\ka}$ that acts on $H$ and $\Pro^1_{\ka}$ as $\iota_1$ and~$\iota_2$ respectively. The quotient $Y = \left( H \times \Pro^1_{\ka} \right) / \langle \iota \rangle$ admits a structure \mbox{of a bundle $Y \rightarrow C \cong H / \langle \iota_1 \rangle$} such that its general fibre is a smooth conic. The images of~fixed points of $\iota$ are $A_1$-singularities. Therefore there are $4k$ singular points on $Y$ lying in~$2k$ fibres of $Y \rightarrow C$ over the~points~$p_i$.

Let $\widetilde{Y} \rightarrow Y$ be the minimal resolution of singularities, and let $\widetilde{Y} \rightarrow S$ be the contraction of~the~proper transforms of the $2k$ fibres of $Y \rightarrow C$ over the points $p_i$. We obtain a conic bundle $S \rightarrow C$ with $2k$ degenerate fibres over the points $p_i$. The components of these fibres are permuted by the group $\Gal\left(\LL / \ka\right)$, therefore $S$ is minimal. The group $\Am(S)$ is~isomorphic to $\mathbb{Z} / 2\mathbb{Z}$ since it is generated by $b(C)$.
\end{example}

The group $\Am(S)$ is a birational invariant, therefore it is obvious that any surface described in case $(1)$ of Proposition \ref{ClassCB} cannot be birationally equivalent to any of~surfaces listed in cases $(2)$--$(5)$. In Section $3$ for the surfaces listed in cases $(2)$--$(4)$ we~show that two of them $S_1$ and $S_2$ are birationally equivalent to each other if and only \mbox{if $\Am(S_1) = \Am(S_2)$.} Surfaces described in case $(5)$ are not birationally equivalent to~surfaces from the other cases.

Note that for a given smooth conic $C$ and even positive integer $2k$ the surface $S$ from case $(4)$ is unique, since it is isomorphic to $\Pro_{C} \left( \mathcal{O}_C \oplus \mathcal{O}_C (-2k) \right)$.

Now we describe minimal del Pezzo surfaces $S$ with $\rho(S) = 1$ and non-trivial $\Am(S)$.

\begin{proposition}
\label{ClassDP}

Let $S$ be a del Pezzo surface over a perfect field $\ka$, such that $\rho(S) = 1$ and $\Am(S)$ is non-trivial. Then there are two possibilities.

\begin{enumerate}

\item The surface $S$ is a non-trivial Severi--Brauer surface, ${K_S^2 = 9}$, \mbox{and $\Am(S) \cong \mathbb{Z} / 3\mathbb{Z}$.}

\item The surface $S$ is a pointless del Pezzo surface of degree $8$ isomorphic to $R_{\LL/\ka} C$, where $\LL$ is a quadratic extension of $\ka$ and $C$ is a conic over $\LL$, such that for any conic $N$ over $\ka$ the conics $C$ and $N_{\LL}$ are not isomorphic. One has~${\Am(S) \cong \mathbb{Z} / 2\mathbb{Z}}$.

\end{enumerate}

\end{proposition}

\begin{proof}

Note that $\rho\left(\overline{S}\right)^{G_{\ka}} = \rho\left(S\right) = 1$. Therefore $\Pic\left(\overline{S}\right)^{G_{\ka}}$ is generated by~a~class~$D$ such that for a certain number $k$ one has $kD = -K_S$, since $-K_S$ obviously lies \mbox{in $\Pic\left(S\right) \subset \Pic\left(\overline{S}\right)^{G_{\ka}}$.}

Assume that there is a $(-1)$-curve $E$ on $\overline{S}$, then one has
$$
1 = -K_{\overline{S}} \cdot E = kD \cdot E = k (D \cdot E).
$$
Therefore $k = 1$, since $D \cdot E$ is integer. Thus in this case $\Pic\left(S\right) = \Pic\left(\overline{S}\right)^{G_{\ka}}$, and $\Am(S)$ is trivial.

If there are no $(-1)$-curves on $\overline{S}$ then $\overline{S}$ is either $\Pro^2_{\kka}$, or $\Pro^1_{\kka} \times \Pro^1_{\kka}$. In the former case~$S$ is~a~del~Pezzo surface of degree $9$, i.e. Severi--Brauer surface. A trivial Severi--Brauer surface is just $\Pro^2_{\ka}$, and $\Am\left(\Pro^2_{\ka}\right)$ is trivial. For a non-trivial Severi--Brauer surface $S$ it~is well-known that $\Pic(S)$ is generated by $-K_S$, and $\Pic\left(\overline{S}\right)^{G_{\ka}} = \Pic\left(\Pro^2_{\kka}\right)^{G_{\ka}} = \Pic\left(\Pro^2_{\kka}\right)$ is~generated by the class of a line on $\Pro^2_{\kka}$. Therefore $\Am(S) \cong \mathbb{Z} / 3\mathbb{Z}$.

If $\overline{S} \cong \Pro^1_{\kka} \times \Pro^1_{\kka}$ then $K_S^2 = 8$. As in the case $\rho(S) = 2$, the group $\Pic\left( \Pro^1_{\kka} \times \Pro^1_{\kka} \right)$ is~generated by the classes $A$ and $B$ of fibres of the projections on the first and the second factors of $\Pro^1_{\kka} \times \Pro^1_{\kka}$, and $\Pic(S)$ contains $-K_S \sim 2A + 2B$. If $\rho(S) = 1$, then the Galois group $G_{\ka}$ permutes $A$ and $B$, and $\Pic\left(\overline{S}\right)^{G_{\ka}}$ is generated by the class $A + B$. Therefore if $\Am(S)$ is not trivial, then $\Am(S) \cong \mathbb{Z} / 2\mathbb{Z}$ and $\Pic(S)$ does not contain $A + B$.

As in the proof of Proposition \ref{ClassCB} one can easily show that $\Pic(S)$ contains $A + B$ if and only if $S$ is isomorphic to a quadric in $\Pro^3_{\ka}$. Therefore for this case we can apply Lemma~\ref{Class1} and get that $S$ is isomorphic to $R_{\LL/\ka} N_{\LL}$, where $N$ is a conic over $\ka$. Thus for \mbox{the case $\Am(S) \cong \mathbb{Z} / 2\mathbb{Z}$} the surface $S$ is not isomorphic to a quadric in $\Pro^3_{\ka}$. By Lemma~\ref{Class1} the surface $S$ is isomorphic to $R_{\LL/\ka} C$, where $\LL$ is a quadratic extension of $\ka$ and $C$ is~a~conic over $\LL$, such that for any conic $N$ over $\ka$ the conics $C$ and $N_{\LL}$ are not isomorphic. 
\end{proof}

\begin{remark}
It is well known, that two non-trivial Severi--Brauer surfaces $S$ and $S'$ are birationally equivalent if and only if $S \cong S'$ or $S' \cong S^{\mathrm{op}}$, where $S^{\mathrm{op}}$ is the Severi--Brauer surface such that the central simple algebras corresponding to $S$ and $S^{\mathrm{op}}$ are opposite (see \cite[Corollary 9.5]{Ami55}). Any other minimal surface is not birationally equivalent to $S$ and~$S'$ by~\cite[Corollary 2.4 and Theorem 2.10]{Sh20} (see also \cite{Wei22}). Applying Propositions \ref{ClassCB} and \ref{ClassDP} we can obtain an~alternative proof of this fact.

By Proposition \ref{ClassDP} a non-trivial Severi--Brauer surface $S$ has $\Am(S) \cong \mathbb{Z} / 3\mathbb{Z}$ generated by the class $b(S)$ of $S$. If a surface $S'$ is birationally equivalent to $S$ \mbox{then $\Am(S') = \Am(S) \cong \mathbb{Z} / 3\mathbb{Z}$.} Therefore $S$ is not birationally equivalent to~any~relatively minimal conic bundle by Proposition \ref{ClassCB}. If $S'$ is a minimal del Pezzo surface then $S'$ is~a~Severi--Brauer surface by Proposition \ref{ClassDP}, and $b(S')$ \mbox{generates $\Am(S') = \Am(S) \cong \mathbb{Z} / 3\mathbb{Z}$.} There are two possibilities: \mbox{either $b(S') = b(S)$} and~$S' \cong S$, or $b(S') = -b(S)$ and $S' \cong S^{\mathrm{op}}$.

\end{remark}

Note that non-isomorphic del Pezzo surfaces of degree $8$ with $\rho(X) = 1$ can have the~same group $\Am(X)$. For example, if $X$ is a pointless quadric in $\Pro^3_{\ka}$ then $\Am(X)$ is~trivial. Therefore we want to consider more invariants for such surfaces.

For a given del Pezzo surface $X$ of degree $8$ with $\rho(X) = 1$ there exists a quadratic extension $\LL / \ka$, such that $X_{\LL} \cong C_1 \times C_2$, where $C_1$ and $C_2$ are smooth conics over $\LL$. The~field~$\LL$ is uniquely determinated by $X$, since $\LL = \kka^{K}$, where~$K$ is~the~kernel of~the~action of $\Gal\left(\kka / \ka\right)$ on $\Pic(\overline{X})$. For simplicity of notation we will say that~$\LL$ is~the~\textit{splitting field} of $X$.

We can consider the group $\Am\left( X_{\LL} \right) \subset \operatorname{Br}\left(\LL\right)$. This group is obviously a birational invariant. By Lemma \ref{Class1} one has $X \cong R_{\LL/\ka} C$, where $C$ is a smooth conic over $\LL$. Let $C'$ be the conjugate conic. Then by Corollary \ref{Class2} one has $X_{\LL} \cong C \times C'$.

Applying Proposition \ref{ClassCB} we can see that there are three possibilities for $\Am(X_{\LL})$:

\begin{itemize}

\item $\Am(X_{\LL})$ is trivial, if and only if $C \cong \Pro^1_{\LL}$;  

\item $\Am(X_{\LL}) \cong \mathbb{Z} / 2\mathbb{Z}$, if and only if $C \cong C'$ and $C$ is non-trivial;  

\item $\Am(X_{\LL}) \cong \left(\mathbb{Z} / 2\mathbb{Z}\right)^2$, if and only if $C$ and $C'$ are not isomorphic.

\end{itemize}

Therefore for a given del Pezzo surface $X$ of degree $8$ with $\rho(X) = 1$ we can consider a~pair of invariants $\left( \LL, \Am(X_{\LL}) \right)$. We show that this pair is a biregular invariant.

\begin{theorem}
\label{BBreg}
Let $X$ and $X'$ be two del Pezzo surfaces of degree $8$ with $\rho(X) = \rho(X') = 1$. Then $X$ and $X'$ have the same splitting field $\LL$ and $\Am(X_{\LL}) = \Am(X'_{\LL})$ if~and only \mbox{if $X \cong X'$}. 

\end{theorem}

\begin{proof}
If $X \cong X'$ then they obviously have the same splitting field $\LL$ \mbox{and $\Am(X_{\LL}) = \Am(X'_{\LL})$}.

Now assume that $X$ and $X'$ have the same splitting field $\LL$ and~$\Am(X_{\LL}) = \Am(X'_{\LL})$ and~show that $X \cong X'$.

Note that $X \cong R_{\LL/\ka} C$, where $b(C) \in \Am(X_{\LL})$. For \mbox{a $\Gal\left(\LL / \ka \right)$-conjugate} conic $C'$ one has $R_{\LL/\ka} C \cong R_{\LL/\ka} C'$ by Lemma \ref{Class1}. Moreover, $X_{\LL} \cong C \times C'$ by Corollary \ref{Class2}.

If $\Am(X_{\LL})$ is trivial then $C \cong C' \cong \Pro^1_{\ka}$. If $\Am(X_{\LL}) \cong \mathbb{Z} / 2\mathbb{Z}$ then there is a unique non-trivial element in $\Am(X_{\LL})$ corresponding to $C \cong C'$.

If $\Am(X_{\LL}) \cong \left(\mathbb{Z} / 2\mathbb{Z}\right)^2$ then there are three non-trivial elements in $\Am(X_{\LL})$. Two of~these elements are $\Gal\left(\LL / \ka \right)$-conjugate since in this case $X_{\LL} \cong C \times C'$, where $C$ and~$C'$ are not isomorphic. Therefore even if the third non-trivial element in $\Am(X_{\LL})$ corresponds to~a~conic~$\widetilde{C}$ then the $\Gal\left(\LL / \ka \right)$-conjugate conic is isomorphic to $\widetilde{C}$. Then \mbox{for $\widetilde{X} = R_{\LL/\ka} \widetilde{C}$} one has $\widetilde{X}_{\LL} \cong \widetilde{C} \times \widetilde{C}$, and $\Am(\widetilde{X}_{\LL}) \cong \mathbb{Z} / 2\mathbb{Z}$. Thus if $\Am(X_{\LL}) \cong \left(\mathbb{Z} / 2\mathbb{Z}\right)^2$ then we can definitely restore the pair of conjugate conics $C$ and $C'$ such that $X_{\LL} \cong C \times C'$ \mbox{and $X \cong R_{\LL/\ka} C \cong R_{\LL/\ka} C'$.}

Therefore in any case $X' \cong R_{\LL/\ka} C \cong X$.
\end{proof}

\begin{question}
Let $C$ be a conic over $\LL$, the conic $C'$ be its $\Gal\left(\LL / \ka \right)$-conjugate \mbox{and $X \cong R_{\LL/\ka} C$.} Assume that $C$ and $C'$ are not isomorphic. Then $\Am(X_{\LL}) \cong \left(\mathbb{Z} / 2\mathbb{Z}\right)^2$. If~the~third non-trivial element in $\Am(X_{\LL})$ corresponds to~a~conic~$\widetilde{C}$, how can one describe the surface $R_{\LL/\ka} \widetilde{C}$?
\end{question}

In Section $4$ we show that the pair $\left( \LL, \Am(X_{\LL}) \right)$ is also a birational invariant.

\begin{remark}
\label{Quadricaction}
Actually, the consideration of such pairs was inspired by \cite[Lemma~3.4]{ShV18}, that originally was obtained to describe the automorphism groups of pointless del~Pezzo surfaces of~degree~$8$. For the case $\rho(X) = 1$ we show connection between the~groups~$\Am(X)$ and~$\Am\left(X_{\LL}\right)$ and the structure of the group $\Aut(X)$ in the following table.


\begin{center}
\begin{tabular}{|c|c|c|c|c|}
\hline
$\Am(X)$ & $\Am(X_{\LL})$ & Quadric & $C$ & $\Aut(X)$ \\
\hline
$\langle \mathrm{id} \rangle \rule[-5pt]{0pt}{18pt}$ & $\langle \mathrm{id} \rangle$ & yes & $C \cong \Pro^1_{\LL}$ & $\operatorname{PGL}_2(\LL) \rtimes \left( \mathbb{Z} / 2\mathbb{Z} \right)$ \\
\hline
$\langle \mathrm{id} \rangle \rule[-5pt]{0pt}{18pt}$ & $\mathbb{Z} / 2\mathbb{Z}$ & yes & $C \cong N_{\LL}$, where $N$ is a conic over $\ka$ & $\Aut\left(C\right) \rtimes \left( \mathbb{Z} / 2\mathbb{Z} \right)$ \\
\hline
$\mathbb{Z} / 2\mathbb{Z} \rule[-5pt]{0pt}{18pt}$ & $\mathbb{Z} / 2\mathbb{Z}$ & no & $C \cong C'$ is not isomorphic to $N_{\LL}$ & $\Aut\left(X\right) / \Aut\left(C\right) \cong \mathbb{Z} / 2\mathbb{Z}$ \\
\hline
$\mathbb{Z} / 2\mathbb{Z} \rule[-5pt]{0pt}{18pt}$ & $\left(\mathbb{Z} / 2\mathbb{Z}\right)^2$ & no & $C$ is not isomorphic to $C'$ & $\Aut\left(X\right) \cong \Aut\left(C\right)$ \\
\hline
\end{tabular}
\end{center}


In the third column we give an answer to the question whether the surface $X$ is isomorphic to a quadric in $\Pro^3_{\ka}$. In the fourth row the group $\Aut\left(X\right)$ is a non-split extension of~$\Aut\left(C\right)$ by $\mathbb{Z} / 2\mathbb{Z}$, that means that $\Aut\left(X\right)$ contains a normal subgroup $\Aut\left(C\right)$ of~index $2$ but is not isomorphic to any semi-direct product $\Aut\left(C\right) \rtimes \left( \mathbb{Z} / 2\mathbb{Z} \right)$. 
\end{remark}

\section{Sarkisov links for the case $\rho(X) = 2$}

In this section we consider Sarkisov links for pointless del Pezzo surfaces of degree $8$ with the Picard number $2$, and prove Theorem \ref{mainrho2}.

Let $X$ be a pointless del Pezzo surface of degree $8$ with the Picard number $2$. By~Lemma~\ref{Class1} in this case $X$ is isomorphic to a product of two smooth conics, such that at least one of these conics is non-trivial. Therefore by Proposition \ref{ClassCB} \mbox{either $\Am(X) \cong \left(\mathbb{Z} / 2\mathbb{Z}\right)^2$,} or $\Am(X) \cong \mathbb{Z} / 2\mathbb{Z}$.

We start from the case $\Am(X) \cong \left(\mathbb{Z} / 2\mathbb{Z}\right)^2$. Assume that there exists a minimal surface~$X'$, and a birational map $X \dashrightarrow X'$. Then $\Am(X') = \Am(X) \cong \left(\mathbb{Z} / 2\mathbb{Z}\right)^2$. By~Propositions~\ref{ClassCB} and \ref{ClassDP} it is possible only if $X' \cong C'_1 \times C'_2$, where $C'_1$ and $C'_2$ are two smooth non-trivial non-isomorphic conics. Moreover, \cite[Theorem 2]{Kol05} implies that if for a~minimal surface $X' \cong C'_1 \times C'_2$ one has $\Am(X') = \Am(X) \cong \left(\mathbb{Z} / 2\mathbb{Z}\right)^2$, then $X'$ is birationally equivalent to $X$.

Actually we want to remind some results and constructions from \cite{Kol05} to describe birational maps from $X$ more explicitly.

\begin{definition}
\label{BPdef}
Let $C_1$ and $C_2$ be two smooth conics, and $b(C_1)$ and $b(C_2)$ be the classes of these conics in $\operatorname{Br}\left(\ka\right)$. If there exists a conic $C_3$, such that $b(C_3) = b(C_1) + b(C_2)$, then~$C_3$ is called the \textit{Brauer product} of $C_1$ and $C_2$. We denote the Brauer product of $C_1$ and $C_2$ by $C_1 * C_2$.
\end{definition}

We want to show that $C_1 * C_2$ is defined if and only if there exists a quadratic extension~$F / \ka$, such that $C_1 \otimes F \cong C_2 \otimes F \cong \Pro^1_{F}$. In the other words there are points of~degree~$2$ on $C_1$ and $C_2$, such that the corresponding geometric points are defined over~$F$. Note that this condition holds if $C_1 \cong C_2$, or if $C_1$ or $C_2$ is trivial.

\begin{lemma}
\label{CBlink}
Let $C_1$ and $C_2$ be two smooth conics. Suppose that there exist points of~degree~$2$ on $C_1$ and $C_2$, such that the corresponding geometric points are defined over a~quadratic extension $F / \ka$. Then there exists a conic $C_3 = C_1 * C_2$, and $X \cong C_1 \times C_2$ is~birationally equivalent to $C_1 \times C_3$ and $C_2 \times C_3$.
\end{lemma}

\begin{proof}
Let $\pi_1: X \rightarrow C_1$ and $\pi_2: X \rightarrow C_2$ be the projections. Let $A$ and $B$ be the classes in $\Pic(\overline{X})$ of geometric fibres of $\pi_1$ and $\pi_2$ respectively. Let $A_1$ and $A_2$ be conjugate fibres of $\pi_1$ defined over $F$, and $B_1$ and $B_2$ be conjugate fibres of $\pi_2$ defined over $F$. Then the~pair of points $A_1 \cap B_1$ and $A_2 \cap B_2$ are defined over $F$. We can equivariantly blow up this pair of points, contract the proper transforms of $A_1$ and $A_2$, and get a surface $X'$.

Note that the birational map $X \dashrightarrow X'$ respects the projection $\pi_1: X \rightarrow C_1$. Therefore~$X'$ admits a structure of a conic bundle over $C_1$. Moreover, $K_{X'}^2 = K_X^2 = 8$, therefore~$X'$ is a $\ka$-form of a Hirzebruch surface $\mathbb{F}_n$. Also note that the proper transforms of~$B_1$ and $B_2$ are sections of the conic bundle $X' \rightarrow C_1$ with selfintersection number $0$. For $n > 0$ there are no sections of $\mathbb{F}_n \rightarrow C_1$ with selfintersection number $0$. Thus $X' \cong \mathbb{F}_0$ that is a product of two conics.

One of these conics is $C_1$, since $X'$ admits a structure of a conic bundle over $C_1$. Denote the other conic by $C_3$. Note that preimage of a geometric point on $C_3$ under the~sequence of~maps~$X \dashrightarrow X' \rightarrow C_3$ has class $A + B$ in $\Pic\left(\overline{X}\right)$. \mbox{Therefore $b(C_3) \in \operatorname{Br}\left(\ka\right)$} \mbox{is $b(C_1) + b(C_2)$,} since the preimages of geometric points on~$C_1$ and~$C_2$ under the projections~$\pi_1$ and~$\pi_2$ have classes $A$ and $B$ in $\Pic\left(\overline{X}\right)$ respectively. Thus $C_3 = C_1 * C_2$, \mbox{and $X \cong C_1 \times C_2$} is birationally equivalent to $X' \cong C_1 \times C_3$. In the same way we can show that $X$ is birationally equivalent to $C_2 \times C_3$.
\end{proof}

The following lemma directly follows from the proof of \cite[Lemma 8]{Kol05}.

\begin{lemma}
\label{CBnonlink}
Let $C_1$ and $C_2$ be two smooth conics, such that for any pair of points of degree $2$ on $C_1$ and $C_2$ the corresponding geometric points are not defined over any quadratic extension $\mathbb{F}/\ka$. Then a conic with class $b(C_1) + b(C_2)$ does not exist.
\end{lemma}

One can find an example of such pair of conics in \cite[Example 5.4]{Cao18}.

Now consider the case $\Am(X) \cong \mathbb{Z} / 2\mathbb{Z}$. By Proposition \ref{ClassCB} the surface $X$ is isomorphic either to $C \times C$, or to $C \times \Pro^1_{\ka}$ where $C$ is a smooth non-trivial conic. Moreover, by~Lemma~\ref{CBlink} for a given smooth non-trivial conic $C$ the~surfaces~$C \times C$ and $C \times \Pro^1_{\ka}$ are birationally equivalent. In the following lemma we show that these surfaces are birationally equivalent to certain $\ka$-forms of Hirzebruch surfaces.

\begin{lemma}
\label{CBLinkF2k}
Let $C$ be a smooth non-trivial conic. Then the surfaces $C \times C$, $C \times \Pro^1_{\ka}$, and $\ka$-forms of Hirzebruch surfaces $\mathbb{F}_{2k}$ admitting a structure of conic bundle over $C$ are~birationally equivalent.
\end{lemma}

\begin{proof}
The surfaces $C \times C$ and $C \times \Pro^1_{\ka}$ are birationally equivalent by Lemma \ref{CBlink}, \mbox{since $C * C \cong \Pro^1_{\ka}$.}

Now consider $X \cong C \times \Pro^1_{\ka}$. Let $\pi_1: X \rightarrow C$ and $\pi_2: X \rightarrow \Pro^1_{\ka}$ be the projections. Then for a $\ka$-point $p \in \Pro^1_{\ka}$, the preimage $B = \pi_2^{-1}(p)$ is a section of $\pi_2$. Let $A_1$ and $A_2$ be conjugate fibres of $\pi_1$. Then the points $A_1 \cap B$ and $A_2 \cap B$ are conjugate. We can equivariantly blow up this pair of points, contract the proper transforms of $A_1$ and $A_2$, and get a surface $X'$.

The birational map $X \dashrightarrow X'$ respects the projection $\pi_1: X \rightarrow C$. Therefore~$X'$ admits a structure of a conic bundle over $C$. Moreover, $K_{X'}^2 = K_X^2 = 8$, therefore~$X'$ is a $\ka$-form of a Hirzebruch surface $\mathbb{F}_n$. Note that the proper transform of $B$ is a section of~the~conic bundle $X' \rightarrow C$ with selfintersection number $-2$. Therefore $n = 2$. \mbox{Thus $C \times \Pro^1_{\ka}$} is~birationally equivalent to a $\ka$-form of a Hirzebruch surface $\mathbb{F}_2$ admitting a structure of a conic bundle over $C$.

Now consider a $\ka$-form $S$ of a Hirzebruch surface $\mathbb{F}_{2k}$ admitting a structure of a conic bundle $\pi:S \rightarrow C$. There is a unique section $H$ of $\pi$ such that $H^2 = -2k$. Let $A_1$ and~$A_2$ be conjugate fibres of $\pi$. Then the points $A_1 \cap H$ and $A_2 \cap H$ are conjugate. We can equivariantly blow up this pair of points, contract the proper transforms of $A_1$ and $A_2$, and get a surface $S'$. 

The birational map $S \dashrightarrow S'$ respects the projection $\pi: S \rightarrow C$. Therefore $S'$ admits a structure of a conic bundle over $C$. Moreover, $K_{S'}^2 = K_S^2 = 8$, therefore $S'$ is a $\ka$-form of a Hirzebruch surface $\mathbb{F}_n$. Note that the proper transform of $H$ is a section of the conic bundle $S' \rightarrow C$ with selfintersection number $-2k-2$. Therefore $n = 2k+2$. Thus $S$ is~birationally equivalent to a $\ka$-form of a Hirzebruch surface $\mathbb{F}_{2k+2}$ admitting a structure of conic bundle over $C$.

Thus we see that all surfaces listed in Lemma \ref{CBLinkF2k} are birationally equivalent.
\end{proof}

Now we prove Theorem \ref{mainrho2}.

\begin{proof}[Proof of Theorem \ref{mainrho2}]

Let $X$ be a pointless del Pezzo surface of degree $8$ \mbox{with $\rho(X) = 2$.} Then $\overline{X} \cong \Pro^1_{\kka} \times \Pro^1_{\kka}$, and $X$ is a product of two conics. If $\Am(X)$ is trivial \mbox{then $X \cong \Pro^1_{\ka} \times \Pro^1_{\ka}$} and $X(\ka) \neq \varnothing$. Therefore $\Am(X)$ is non-trivial, and by Proposition~\ref{ClassCB} \mbox{either $\Am(X) \cong \left( \mathbb{Z} / 2\mathbb{Z} \right)^2$,} or $\Am(X) \cong \mathbb{Z} / 2\mathbb{Z}$.

If $\Am(X) \cong \left( \mathbb{Z} / 2\mathbb{Z} \right)^2$ then $X$ is a product of two non-trivial smooth conics $C_1$ and~$C_2$, not isomorphic to each other. If the Brauer product $C_3 = C_1 * C_2$ is defined \mbox{then $X \approx C_1 \times C_3 \approx C_2 \times C_3$} by Lemma \ref{CBlink}. 

If a minimal surface $X'$ is birationally equivalent to $X$ then
$$
\Am(X') = \Am(X) \cong \left( \mathbb{Z} / 2\mathbb{Z} \right)^2.
$$
\noindent By Propositions \ref{ClassCB} and \ref{ClassDP} this is possible only if $X'$ is a product of two smooth conics $C'_1$ and $C'_2$. Moreover, $b(C'_1)$ and $b(C'_2)$ are distinct non-trivial elements of $\Am(X)$. Therefore if $C_3$ is defined then $X'$ is isomorphic to $C_1 \times C_2$, $C_1 \times C_2$ or $C_2 \times C_3$, \mbox{and otherwise $X' \cong C_1 \times C_2 \cong X$.} We have proved Theorem \ref{mainrho2}(1).

If $\Am(X) \cong \mathbb{Z} / 2\mathbb{Z}$ then either $X \cong C \times C$, or $X \cong C \times \Pro^1_{\ka}$, where $C$ is a non-trivial smooth conic. By Lemma \ref{CBLinkF2k} the surface $X$ is birationally equivalent to $C \times C$, $C \times \Pro^1_{\ka}$, and $\ka$-forms of Hirzebruch surfaces $\mathbb{F}_{2k}$ admitting a structure of conic bundle over $C$.

Assume that a minimal surface $X'$ is birationally equivalent to $X$ and not listed before. Then the map $X \dashrightarrow X'$ can be decomposed into a sequence of Sarkisov links:
$$
X = X_0 \dashrightarrow X_1 \dashrightarrow \ldots \dashrightarrow X_{n-1} \dashrightarrow X_n = X'.
$$
In particular there exists a Sarkisov link $f: S \dashrightarrow S'$, such that $S$ is listed in~Lemma~\ref{CBLinkF2k}, and $S'$ is not listed. For geometrically rational surfaces Sarkisov links are described in~\cite[Theorem 2.6]{Isk96}. All surfaces listed in Lemma \ref{CBLinkF2k} admit a structure of a conic bundle over $C$. Therefore $f$ has type $\mathrm{II}$, $\mathrm{III}$ or $\mathrm{IV}$. Links of types $\mathrm{II}$ and $\mathrm{IV}$ preserve~$K_S^2$ and~$\Am(S)$. Therefore in this case $S'$ is one of the surfaces listed \mbox{in Proposition \ref{ClassCB}(2)--(4),} and also in Lemma \ref{CBLinkF2k}. For $S$ with $\rho(S) = 2$ and $K_S^2 = 8$ the only possible link of type~$\mathrm{III}$ is a contraction of $(-1)$-section on a $\ka$-form of a Hirzebruch surface $\mathbb{F}_1$. But such surface is not listed in Lemma \ref{CBLinkF2k}, therefore this link is~impossible, and $X'$ must be isomorphic to a surface listed in Lemma \ref{CBLinkF2k}. We have proved Theorem \ref{mainrho2}(2).
\end{proof}

\section{Sarkisov links for the case $\rho(X) = 1$}

In this section we consider Sarkisov links for pointless del Pezzo surfaces of degree $8$ with the Picard number $1$, and prove Theorem \ref{mainrho1}.

Let $X$ be a pointless del Pezzo surface of degree $8$ with the Picard number $1$. By~Lemma~\ref{Class1} in this case $X$ is isomorphic to $R_{\LL/\ka} C$, where $\LL$ is a quadratic extension of~$\ka$ and $C$ is a conic over $\LL$. The field $\LL$ is uniquely determined by~$X$. \mbox{The pair $\left( \LL, \Am(X_{\LL}) \right)$} is a biregular invariant of $X$ by Theorem \ref{BBreg}. We want to find possibilities for minimal surfaces $X'$ birationally equivalent to $X$.

Note that by Corollary \ref{Class2} one has $X_{\LL} \cong C \times C'$, where $C'$ is a conic conjugate to $C$ under the action of $\Gal \left( \LL / \ka \right)$. Note that $X$ is pointless if and only if $X_{\LL}$ is pointless too. In particular, if $X$ is pointless then $\Am(X_{\LL})$ is non-trivial.

\begin{lemma}
\label{DPpoints}
Let $X$ be a pointless del Pezzo surface of degree $8$ with $\rho(X) = 1$. Then any point on $X$ has even degree.
\end{lemma}

\begin{proof}
Assume that there exists a point of odd degree on $X$. Then there exists a point of~odd degree on $X_{\LL}$, and images of this point under the projection on $C$ and $C'$ are points of odd degree. There are no points of odd degree on any non-trivial conic. \mbox{Therefore $C \cong C' \cong \Pro^1_{\LL}$,} and thus $X\left( \ka \right) \neq \varnothing$.

We obtain a contradiction. Therefore any point on $X$ has even degree.
\end{proof}

Now we want to consider Sarkisov links for the surface $X$. These links are described in~\cite[Theorem 2.6 (i and ii)]{Isk96}. In our case such a link must have tipe $\mathrm{I}$ or $\mathrm{II}$, since~$X$ is~a~del~Pezzo surface with $\rho(X) = 1$. Any link of type $\mathrm{I}$ is a blowup $X_1 \rightarrow X$ of~a~point of~degree~$d$, where $X_1$ has a structure of a conic bundle over a smooth conic and $\rho(X_1) = 2$. Any link of type $\mathrm{II}$ is a composition $\sigma_1 \circ \sigma^{-1}$, where $\sigma: Z \rightarrow X$ is a blowup of~a~point of~degree~$d$, and $\sigma_1: Z \rightarrow X_1$ is a contraction of a set of disjoint conjugate $(-1)$-curves, that differs from the exceptional divisor of $\sigma$. In our case $d = 2$, $d = 4$ or $d = 6$ by~Lemma~\ref{DPpoints}. For $d = 2$ the corresponding link has type $\mathrm{I}$, and for $d = 4$ and $d = 6$ the~corresponding link has type $\mathrm{II}$. For each of these three links we want to describe $X_1$.

\begin{lemma}
\label{DPlink6}
Let $X$ be a del Pezzo surface of degree $8$ with $\rho(X) = 1$. Let $X \dashrightarrow X_1$ be~a~Sarkisov link corresponding to the blowup of a point of degree $6$. Then $X_1 \cong X$.
\end{lemma}

\begin{proof}
This case is considered in \cite[Theorem 2.6.(ii), $K_X^2 = 8$, $d = 6$]{Isk96}, and there is~explicitely written that $X_1 \cong X$.
\end{proof}

\begin{lemma}
\label{DPlink4}
Let $X$ be a del Pezzo surface of degree $8$ with $\rho(X) = 1$. Let $X \dashrightarrow X_1$ be~a~Sarkisov link corresponding to the blowup of a point of degree $4$. Then $X_1 \cong X$.
\end{lemma}

\begin{proof}
This case is considered in \cite[Theorem 2.6.(ii), $K_X^2 = 8$, $d = 4$]{Isk96}, and for this case $K_{X_1}^2 = 8$ and $\rho(X_1) = 1$. Therefore by Lemma \ref{Class1} one has $R_{\LL_1 / \ka} C_1$, where $\LL_1$ is~the~splitting field of $X_1$ and $C_1$ is a smooth conic over $\LL_1$. Let us show that $\LL_1 = \LL$. It~is~sufficient to prove that $\rho(\left(X_1\right)_{\LL}) = 2$.

Let $\pi_1: X_{\LL} \rightarrow C$ and $\pi_2: X_{\LL} \rightarrow C'$ be the projections. Let $A$ and $B$ be the classes in~$\Pic(\overline{X})$ of geometric fibres of $\pi_1$ and $\pi_2$ respectively. The map $X \dashrightarrow X_1$ is a composition $\sigma_1 \circ \sigma^{-1}$, where $\sigma: Z \rightarrow X$ is a blowup of a point of degree $4$ and $\sigma_1: Z \rightarrow X_1$ is~a~contraction of a set of disjoint conjugate $(-1)$-curves. Let $E_1$, $E_2$, $E_3$ and $E_4$ be~the~classes in $\Pic(\overline{Z})$ of geometrically irreducible components of the exceptional divisor of $\sigma$.

The surface $Z$ is a del Pezzo surface of degree $4$, and thus there are $16$ $(-1)$-curves on~$Z$, having classes $E_i$, $A-E_i$, $B-E_i$ and $A+B+E_i-\sum \limits_{j = 1}^4 E_j$ in $\Pic(\overline{Z})$. \mbox{The four $(-1)$-curves} with classes $A+B+E_i-\sum \limits_{j = 1}^4 E_j$ are conjugate and disjoint, therefore $\sigma_1:Z \rightarrow X_1$ is~the~contraction of these curves.

Let $k$ be the number of $\Gal\left( \kka / \LL \right)$-orbits on the set $E_i$ (actually $k = 1$ or $k = 2$). Then the number of $\Gal\left( \kka / \LL \right)$-orbits on the set $A+B+E_i-\sum \limits_{j = 1}^4 E_j$ is also $k$, since $A+B-\sum \limits_{j = 1}^4 E_j$ is $\Gal\left( \kka / \LL \right)$-invariant. Therefore
$$
\rho(\left(X_1\right)_{\LL}) = \rho(Z_{\LL}) - k  = \rho(X_{\LL}) = 2.
$$

The del Pezzo surfaces $X$ and $X_1$ of degree $8$ with $\rho(X) = \rho(X_1) = 1$ have the same splitting field and $\Am(X_{\LL}) = \Am((X_1)_{\LL})$, since $X \approx X_1$. Therefore $X \cong X_1$ by Theorem~\ref{BBreg}.

\end{proof}

\begin{remark}
Note that in Lemmas \ref{DPlink6} and \ref{DPlink4} the surface $X$ is not necessary pointless.
\end{remark}

The link of type $\mathrm{I}$ corresponding to the blowup $Z \rightarrow X$ of two conjugate geometric points is considered in \cite[Theorem 2.6.(i), $K_X^2 = 8$]{Isk96}. In this case $Z$ admits a~structure of a conic bundle $\pi: Z \rightarrow C$ with two degenerate geometric fibres, where~$C$ is~a~conic such that $b(C)$ is the generator of $\Am(X)$ (in particular, $C \cong \Pro^1_{\ka}$ if $\Am(X)$ is~trivial). In~the~following proposition we collect some facts about such conic bundles following~\cite{Isk79} and~\cite{Isk96}.

\begin{proposition}
\label{CB6}
Let $\pi: S \rightarrow C$ be a conic bundle over a smooth conic $C$ with two degenerate geometric fibres, such that $\rho(S) = 2$. Then $K_S^2 = 6$, and $S$ has the following properties.

\begin{enumerate}

\item[(i)] The surface $S$ is a del Pezzo surface of degree $6$ (see \cite[Theorem 5]{Isk79}).

\item[(ii)] There are six $(-1)$-curves on $S$: four components of the singular fibres, and two disjoint conjugate sections $E_1$ and $E_2$ of $\pi$. These curves form a hexagon.

\item[(iii)] The surface $S$ is not minimal (see \cite[Theorem 4]{Isk79}).

\item[(iv)] There exists a unique link $S \rightarrow Y$ of type $\mathrm{III}$ that is the contraction of~the~pair~$E_1$ and $E_2$. The surface $Y$ is a del Pezzo surface of degree $8$ with $\rho(Y) = 1$ (see \cite[Theorem 2.6(iii)]{Isk96}). 

\item[(v)] Any other link $S \dashrightarrow S_1$ has type $\mathrm{II}$, and preserves a structure of a conic bundle over~$C$ (see \cite[Theorem 2.6]{Isk96}). Such link is a composition $\sigma_1 \circ \sigma^{-1}$, where~$\sigma$ is~a~blowup of a point of degree $d$ such that the corresponding geometric points do~not~lie on the degenerate fibres, and $\sigma_1$ is the contraction of~the~proper transforms of $d$ fibres containing these points. In particular, the conic bundles~$S \rightarrow C$ and~$S_1 \rightarrow C$ have degenerate fibres over the same points on $C$, $\rho(S_1) = 2$ and~$K_{S_1}^2 = 6$. Therefore all described above properties hold for $S_1$.

\end{enumerate}

\end{proposition}

Now we can consider the link of type $\mathrm{I}$.

\begin{lemma}
\label{DPlink2}
Let $X$ be a pointless del Pezzo surface of degree $8$ with $\rho(X) = 1$. Consider a sequence of Sarkisov links
$$
X \leftarrow Z = Z_0 \dashrightarrow Z_1 \dashrightarrow \ldots \dashrightarrow Z_{n-1} \dashrightarrow Z_n = Z' \rightarrow X',
$$ 
where the first link has type $\mathrm{I}$, the last link has type $\mathrm{III}$, and all other links have type $\mathrm{II}$. Then $X' \cong X$.

\end{lemma}

\begin{proof}
The surface $Z$ admits a structure of a conic bundle over a smooth conic $\pi:Z \rightarrow C$ with two degenerate fibres. Denote the components of these fibres by $A_1$, $A_2$, $B_1$ and $B_2$. These components are $(-1)$-curves on $Z$. By Proposition \ref{CB6}(ii) there are two disjoint conjugate sections $E_1$ and $E_2$ of $\pi$ that are $(-1)$-curves. Without loss of generality we~can assume that
$$
A_1 \cdot E_1 = E_1 \cdot B_1 = B_1 \cdot A_2 = A_2 \cdot E_2 = E_2 \cdot B_2 = B_2 \cdot A_1 = 1.
$$
Let $\Gamma$ be the kernel of the action of $\Gal\left(\kka / \ka\right)$ on the set of $(-1)$-curves on $Z$, and $\mathbb{F} = \kka^{\Gamma}$. Then $\rho\left(Z_{\mathbb{F}}\right) = \rho\left(\overline{Z}\right) = 4$, and $\rho(X_{\mathbb{F}}) = 2$. Therefore $\LL$ is a subfield of $\mathbb{F}$. Note that the~pairs $A_1$ and $A_2$, $B_1$ and $B_2$, and $E_1$ and $E_2$ are defined over $\LL$, but no one of these curves is defined over $\LL$, since $X_{\LL}$ is pointless. Therefore the group $\Gal\left(\mathbb{F} / \LL \right)$ has order~$2$ and pairwisely permutes $E_1$ and $E_2$, $A_1$ and $A_2$, and $B_1$ and $B_2$.

By Proposition \ref{CB6}(v) all surfaces $Z_i$ are del Pezzo surfaces of degree $6$ admitting a~structure of a conic bundle $Z_i \rightarrow C$ with two degenerate fibres. Let $f: Z \dashrightarrow Z'$ be the composition of the Sarkisov links $Z_i \dashrightarrow Z_{i+1}$. Then $A'_1 = f(A_1)$, $A'_2 = f(A_2)$, $B'_1 = f(B_1)$ and~$B'_2 = f(B_2)$ are components of the degenerate fibres of $Z' \rightarrow C$. By~Proposition \ref{CB6}(ii) there are two disjoint conjugate sections $E'_1$ and $E'_2$ of $\pi': Z' \rightarrow C$ that are $(-1)$-curves. We can assume that $A'_1 \cdot E'_1 = 1$, and want to show that $E'_1 \cdot B'_1 = 1$.

Let $R = f^{-1}_*\left(E'_1\right)$ be the proper transform of $E'_1$ on $Z$. If we blow up a point $p$ on $\overline{Z}$ and contract a proper transform of the fibre containing $p$, then for the proper transform~$\widetilde{R}$ of~$R$ one has $\widetilde{R}^2 - R^2 = \pm 1$ ($-1$ if $p \in R$, and $1$ otherwise). By Proposition~\ref{CB6}(v) any link $Z_i \dashrightarrow Z_{i+1}$ is a composition $\sigma_1 \circ \sigma^{-1}$, where $\sigma$ is a blowup of a point of~degree~$d$ such that the corresponding geometric points do not lie on degenerate fibres, and $\sigma_1$ is~the~contraction of the proper transforms of $d$ fibres containing these points. Moreover, $d$ is even, since otherwise there is a point of odd degree on $X$ that is impossible by~Lemma~\ref{DPpoints}. Thus $R^2 - \left(E'_1\right)^2$ is even, and $R^2$ is odd.

Note that the group $\Pic(\overline{Z}) \cong \mathbb{Z}^4$ is generated by $A_1$, $E_1$, $B_1$ and the class \mbox{of a fibre $F \sim A_1 + B_2 \sim A_2 + B_1$.} Therefore $R \sim aA_1 + eE_1 + bB_1 + fF$. The~curve~$R$ is a section of $\pi: Z \rightarrow C$, thus $e = 1$, since $F^2 = F \cdot A_1 = F \cdot B_1 = 0$ and $F \cdot E_1 = 1$. One \mbox{has $A_1 \cdot R = A'_1 \cdot E'_1 = 1$,} thus $a = 0$, since $A_1 \cdot B_1 = A_1 \cdot F =0$, $A_1 \cdot E_1 = 1$ and~$A_1^2=-1$. Therefore $R \sim E_1 + bB_1 + fF$, and
$$
R^2 = E_1^2 + b^2B_1^2 +f^2F^2 +2bE_1 \cdot B_1 + 2fE_1 \cdot F + 2bfB_1 \cdot F = -1 - b^2 + 2b +2f.
$$
The number $b$ is even, since the intersection number $R^2$ is odd. One has
$$
B_1 \cdot R = B_1 \cdot \left( E_1 + bB_1 + fF \right) = 1 - b.
$$
Therefore $b = 0$, since $b$ is even and $B_1 \cdot R \geqslant 0$. Thus $E'_1 \cdot B'_1 = R \cdot B_1 = 1$.

Now we see that for the set of $(-1)$-curves on $Z'$ we have
$$
A'_1 \cdot E'_1 = E'_1 \cdot B'_1 = B'_1 \cdot A'_2 = A'_2 \cdot E'_2 = E'_2 \cdot B'_2 = B'_2 \cdot A'_1 = 1.
$$
The group $\Gal\left(\mathbb{F} / \LL \right)$ has order $2$ and pairwisely permutes $A'_1$ and $A'_2$, $B'_1$ and $B'_2$, and~also~$E'_1$ and $E'_2$. Therefore $\rho\left(Z'_{\LL}\right) = 3$. By Proposition \ref{CB6}(iv) the map $Z' \rightarrow X'$ is~the~contraction of the disjoint conjugate curves $E'_1$ and $E'_2$. Each of those curves is~not~defined over $\LL$, therefore $\rho\left(X'_{\LL}\right) = 2$. It means that $\LL$ is the splitting field of~$X'$. One has $\Am(X_{\LL}) = \Am(X'_{\LL})$, since $X \approx X'$. Therefore $X \cong X'$ by~Theorem~\ref{BBreg}.
\end{proof}

\begin{remark}

Note that Lemma \ref{DPlink2} does not work if $X(\ka) \neq \varnothing$. Consider a sequence of~Sarkisov links 
$$
X \leftarrow Z \dashrightarrow Z' \rightarrow X',
$$ 
where the first link has type $\mathrm{I}$ and the last link has type $\mathrm{III}$. The second link corresponds to the blowup of a $\ka$-point $p$ on a smooth fibre of $Z \rightarrow C$ and the contraction of~the~proper transform of the fibre passing through $p$, and has type $\mathrm{II}$. Let us show that $X$ and $X'$ are not isomorphic to each other. We use the notation established in the proof of Lemma \ref{DPlink2} for the $(-1)$-curves on $Z$ and $Z'$. We want to show that $E'_1 \cdot B'_1 = 0$.

Let $R$ be the proper transform of $E'_1$ on $Z$. Then $R^2$ is even, and $R \sim E_1 + bB_1 + fF$. Thus
$$
R^2 = E_1^2 + b^2B_1^2 +f^2F^2 +2bE_1 \cdot B_1 + 2fE_1 \cdot F + 2bfB_1 \cdot F = -1 - b^2 + 2b +2f.
$$
The number $b$ is odd, since the intersection number $R^2$ is even. One has
$$
B_1 \cdot R = B_1 \cdot \left( E_1 + bB_1 + fF \right) = 1 - b.
$$
Therefore $b = 1$, since $b$ is odd and $B_1 \cdot R \geqslant 0$. Thus $E'_1 \cdot B'_1 = R \cdot B_1 = 0$.

Now we see that for the set of $(-1)$-curves on $Z'$ we have
$$
A'_1 \cdot E'_1 = E'_1 \cdot A'_2 = A'_2 \cdot B'_1 = B'_1 \cdot E'_2 = E'_2 \cdot B'_2 = B'_2 \cdot A'_1 = 1.
$$
The group $\Gal\left(\mathbb{F} / \LL \right)$ has order $2$ and pairwisely permutes $A'_1$ and $A'_2$, $B'_1$ and $B'_2$, and~the~$(-1)$-curves $E'_1$ and $E'_2$ are defined over $\LL$. Therefore $\rho\left(Z'_{\LL}\right) = 3$. By Proposition~\ref{CB6}(iv) the map $Z' \rightarrow X'$ is the contraction of the disjoint conjugate curves $E'_1$ and~$E'_2$. Each of those curves is defined over $\LL$, therefore $\rho\left(X'_{\LL}\right) = 1$, that means that~$\LL$ is not the splitting field of $X'$. Hence $X$ and $X'$ are not isomorphic to each other by~Theorem~\ref{BBreg}.

\end{remark}

As a by-product of Lemma \ref{DPlink2} we obtain the following proposition.

\begin{proposition}
\label{DP6rigid}
Let $Z$ be a pointless del Pezzo surface of degree $6$ admitting a structure of a relatively minimal conic bundle $Z \rightarrow C$, and $Z \dashrightarrow Z_1$ be a Sarkisov link of type $\mathrm{II}$. Then $Z \cong Z_1$.
\end{proposition}

We decompose the proof into several lemmas that are interesting themselves.

\begin{lemma}
\label{Conictrans}
Let $C$ be a conic over a perfect field $\LL$. Then the group $\Aut(C)$ transitively acts on the set of points of degree $2$, such that the corresponding geometric points are defined over a quadratic extension $\LL\left(\zeta\right)$.
\end{lemma}

\begin{proof}
Assume that $\zeta$ satisfies the equation $\zeta^2 + a \zeta + b = 0$, and $\bar{\zeta}$ be the other root of~this equation.

Let $\left(p_1, p_2\right)$ and $\left(q_1, q_2\right)$ be two pairs of conjugate points on $C$ such that the~points~$p_1$, $p_2$, $q_1$ and $q_2$ are defined over $\LL\left(\zeta\right)$. Consider an embedding $C \hookrightarrow \Pro^2_{\LL}$. The two lines passing through $\left(p_1, p_2\right)$ and $\left(q_1, q_2\right)$ are defined over $\LL$. Therefore we~can choose coordinates in~$\Pro^2_{\LL}$ such that these lines are given by $x = 0$ and $y = 0$, and~the~homogeneous coordinates of~$p_1$, $p_2$, $q_1$ and $q_2$ are
$$
(0: \zeta: 1), \qquad (0: \bar{\zeta}: 1), \qquad (\zeta: 0 : 1), \qquad (\bar{\zeta}: 0 : 1)
$$
\noindent respectively. Then $C$ is given by the equation
$$
x^2 + y^2 + axz + ayz + bz^2 = \theta xy
$$
for some $\theta \in \LL$, and the automorphism $x \leftrightarrow y$ acts on $C$ and permutes the pairs $\left(p_1, p_2\right)$ and $\left(q_1, q_2\right)$.
\end{proof}

\begin{lemma}
\label{Quadrictrans}
Let $S$ be a pointless del Pezzo surface of degree $8$ with $\rho(S) = 1$ over a~perfect field $\ka$. Then the group $\Aut(S)$ transitively acts on the set of points of~degree~$2$, such that the corresponding geometric points are defined over a quadratic extension~$\ka\left(\zeta\right)$.

\end{lemma}

\begin{proof}
By Lemma \ref{Class1} one has $S \cong R_{\LL/\ka} C$, where $\LL$ is the splitting field of $S$ and $C$ is~a~conic over $\LL$. Let $\left(p_1, p_2\right)$ and $\left(q_1, q_2\right)$ be two pairs of conjugate points on $C$ such that the~points~$p_1$, $p_2$, $q_1$ and $q_2$ are defined over $\ka\left(\zeta\right)$. Note that $\zeta \notin \LL$ since $S_{\LL}$ is~pointless by Corollary \ref{Class2}.

By Corollary \ref{Class2} one has $S_{\LL} \cong C \times C'$, where $C'$ is the $\Gal\left(\LL / \ka\right)$-conjugate conic of~$C$. Let $\pi_1:S_{\LL} \rightarrow C$ be the projection. One has $\Aut^0\left(S\right) \cong \Aut(C)$ by~Remark~\ref{Quadricaction}, \mbox{and the group $\Aut^0\left(S\right)$} faithfully acts on $C$, since $S \cong R_{\LL/\ka} C$. Consider \mbox{an element $g \in \Aut^0\left(S\right)$} such that the induced action of $g$ on $C$ maps the~pair~$\left(\pi_1(p_1), \pi_1(p_2)\right)$ to the pair $\left(\pi_1(q_1), \pi_1(q_2)\right)$ (such $g$ exists by Lemma \ref{Conictrans}).

Note that the $\Gal\left(\LL\left(\zeta\right) / \ka \right)$-orbit of the fibre $F = \pi_1^{-1}(\pi_1(q_1))$ consists of the two fibres of $S_{\LL} \rightarrow C$ passing through $q_1$ and $q_2$, and the two fibres of $S_{\LL} \rightarrow C'$ passing through~$q_1$ and $q_2$. These fibres have exactly four common geometric points: the conjugate points $q_1$ and $q_2$, and the two other conjugate points $r_1$ and $r_2$. Note that $r_1$ and $r_2$ are not defined over $\ka\left(\zeta\right)$, since otherwise $F$ is defined over $\ka\left(\zeta\right)$ and $S$ splits over $\ka\left(\zeta\right) \neq \LL$. Therefore the pair $\left(g(p_1), g(p_2)\right)$ coincides with the pair~$\left(q_1, q_2\right)$, and~we are done.
\end{proof}

Now we can prove Proposition \ref{DP6rigid}.

\begin{proof}[Proof of Proposition \ref{DP6rigid}]

Consider the contractions of the negative sections $Z \rightarrow X$ and~${Z_1 \rightarrow X_1}$ (see Proposition \ref{CB6}(iv)). The del Pezzo surfaces $X$ and $X_1$ are isomorphic by Lemma \ref{DPlink2}. By Lemma \ref{Quadrictrans} there exists an automorphism of $X$ that maps the points of the blowup of $Z \rightarrow X$ to the points of the blowup $Z_1 \rightarrow X_1$. Therefore~${Z \cong Z_1}$.
\end{proof}

Now we prove Theorems \ref{mainrho1} and \ref{rigidity}, and give an alternative proof of Theorem \ref{main1} for~the~case of a perfect field.

\begin{proof}[Proof of Theorem \ref{mainrho1}]

Let $X$ be a pointless del Pezzo surface of degree $8$ with $\rho(X) = 1$. Assume that a minimal surface $X''$ is birationally equivalent to $X$. Then $X \dashrightarrow X''$ can~be decomposed into a sequence of Sarkisov links:
$$
X = X_0 \dashrightarrow X_1 \dashrightarrow \ldots \dashrightarrow X_{n-1} \dashrightarrow X_n = X''.
$$

By Lemma \ref{DPpoints} any Sarkisov link corresponds to the blowup of $X$ at a point of~degree~$6$, $4$ or $2$. For the cases of degree $6$ and $4$ this link immediately gives a surface $X_1 \cong X$ by~Lemmas~\ref{DPlink6} and \ref{DPlink4} respectively. For the case of degree $2$ the sequence of links passes through some non-minimal del Pezzo surfaces of degree $6$ (see Proposition \ref{CB6}(iii)) and~terminates at a del Pezzo surface $X' \cong X$ by Lemma \ref{DPlink2}. Therefore any minimal del~Pezzo surface, obtained in a sequence of Sarkisov links starting from $X$, is isomorphic to $X$. In particular, $X'' \cong X$.
\end{proof}

\begin{proof}[Proof of Theorem \ref{rigidity}]

Note that by \cite[Theorem 5.1]{Lied17} there is an embedding $X \hookrightarrow P$, where $P$ is a Severi--Brauer threefold such that the class $b(P)$ lies in $\Am(X) \subset \operatorname{Br}(\ka)$.

One can easily find a point of degree $4$ on $X$, therefore $I(X)$ divides $4$. Moreover, $I(P)$ divides $I(X)$ since any point on $X$ is also a point on $P$. There is a point of degree $I(P)$ on $P$ by \cite[Theorem 53]{Kol16}.

Assume that $I(P) \leqslant 2$, then there is a point of degree $2$ on $P$. The line~$\overline{L}$ passing through the pair of the corresponding points in $\Pro^3_{\kka} = \overline{P}$ is invariant under the action of~the~Galois group $\Gal\left(\kka / \ka\right)$. Therefore there is a smooth curve $L \subset P$ of degree $1$. One has $L \cdot X = 2$. Thus there is a point of degree $2$ on $X$, and $I(X) = 2$ or $I(X) = 1$.

Note that the classes $b(L)$ and $b(P)$ coincide in $\operatorname{Br}(\ka)$ by \cite[Definition-Lemma 31]{Kol16}. Moreover these classes coincide with the image of the class $-\frac{1}{2}K_X$ in $\Am(X)$. Therefore if~$\rho(X) = 1$ then any element of $\Am(X)$ corresponds to a conic since $\Am(X)$ is generated by the image of the class $-\frac{1}{2}K_X$. Moreover, in this case $X$ is not birationally rigid since one can blow up a point of degree $2$ and get a del Pezzo surface of degree $6$ with a structure of a relatively minimal conic bundle. If $\rho(X) = 2$ then any element of $\Am(X)$ corresponds to a conic, and $X$ is not birationally rigid by Lemma \ref{CBlink}.

Now assume that $I(P) = 4$, then $P$ cannot contain a conic by Theorem \cite[Theorem~53]{Kol16}, and the element $b(P) \in \Am(X)$ does not correspond to a conic. Moreover, one has $I(X) = 4$, and there are no points of degree $2$ on $X$ since $I(P)$ divides $I(X)$.

If $\rho(X) = 2$ then $X \cong C_1 \times C_2$ and $\Am(X) \cong \left(\mathbb{Z} / 2\mathbb{Z}\right)^2$. Any minimal surface $X'$ birationally equivalent to $X$ must be a product $C'_1 \times C'_2$ of two non-isomorphic smooth non-trivial conics by Propositions \ref{ClassCB} and \ref{ClassDP}. Therefore $b(C'_1)$ and $b(C'_2)$ are not trivial and are not equal to $b(P)$. Thus the sets $\{b(C'_1), b(C'_2)\}$ and $\{b(C_1), b(C_2)\}$ coincide, and~$X' \cong X$. Hence $X$ is birationally rigid.

If $\rho(X) = 1$ then the only possible Sarkisov link is described in Lemma \ref{DPlink4}, and~therefore~$X$ is birationally rigid.
\end{proof}

\begin{proof}[Proof of Theorem \ref{main1} for the case of a perfect field]

Consider two pointless birationally equivalent quadric surfaces $Q_1$ and $Q_2$. These surfaces are minimal. If $\rho(Q_1) = 1$ then by~Theorem \ref{mainrho2} one has $Q_1 \cong Q_2$.

If $\rho(Q_1) = 2$ then by Lemma \ref{Class1} one has $Q_1 \cong C_1 \times C_1$. By Theorem \ref{mainrho2} any minimal surface birationally equivalent to $Q_1$ has Picard number $2$, and therefore $Q_2 \cong C_2 \times C_2$ by Lemma \ref{Class1}. Thus by Theorem \ref{mainrho2}(2) one has $Q_1 \cong Q_2$.
\end{proof}

\bibliographystyle{alpha}

\end{document}